\documentclass[journal,twoside,web]{ieeecolor}
\pdfoutput=1
\usepackage{lcsys}
\usepackage{cite}
\usepackage{amsmath,amssymb,amsfonts}
\usepackage{mathtools}
\usepackage{graphicx}
\usepackage{textcomp}
\usepackage{booktabs}
\usepackage{hyperref}
\def\BibTeX{{\rm B\kern-.05em{\sc i\kern-.025em b}\kern-.08em
    T\kern-.1667em\lower.7ex\hbox{E}\kern-.125emX}}
\newif\ifshowchanges

\showchangesfalse 

\ifshowchanges
  \usepackage[switch,pagewise]{lineno}
\fi

\ifshowchanges
  \newcommand{\rev}[1]{\textcolor{blue}{#1}}
\else
  \newcommand{\rev}[1]{#1}
\fi

\newtheorem{theorem}{Theorem}[section]

\newtheorem{definition}[theorem]{Definition}
\newtheorem{assumption}[theorem]{Assumption}
\newtheorem{remark}[theorem]{Remark}

\newtheorem{proposition}[theorem]{Proposition}

\newcommand{\Rn}{\mathbb{R}^n}
\newcommand{\bx}{\mathbf{x}}
\newcommand{\bg}{\mathbf{g}}
\newcommand{\bw}{\mathbf{w}}
\newcommand{\bH}{\mathbf{H}}

\newcommand{\bu}{\mathbf{u}}
\newcommand{\cX}{\mathcal{X}}
\newcommand{\cS}{\mathcal{S}}

\begin{document}

\ifshowchanges\linenumbers\fi

\title{Saddle Point Evasion via Curvature-Regularized Gradient Dynamics}

\author{Liraz Mudrik, Isaac Kaminer, Sean Kragelund, and Abram H. Clark
\thanks{This work was supported in part by the Office of Naval Research Science of Autonomy Program under Grant No.\ N0001425GI01545 and Consortium for Robotics Unmanned Systems Education and Research at the Naval Postgraduate School.}
\thanks{L. Mudrik, I. Kaminer, and S. Kragelund are with the Department of Mechanical and Aerospace Engineering, Naval Postgraduate School, Monterey, CA, 93943, USA (e-mail: liraz.mudrik.ctr@nps.edu; kaminer@nps.edu; spkragel@nps.edu).}
\thanks{A. H. Clark is with the Department of Physics, Naval Postgraduate School, Monterey, CA, 93943, USA (e-mail: abe.clark@nps.edu).}}

\maketitle
\thispagestyle{empty}

\begin{abstract}
Nonconvex optimization underlies many modern machine learning and control tasks, where saddle points pose the dominant obstacle to reliable convergence in high-dimensional settings. Escaping these saddle points deterministically \rev{using continuous-time optimization} remains an open challenge: gradient descent is blind to curvature, stochastic perturbation methods lack deterministic guarantees, and Newton-type approaches suffer from Hessian singularity. \rev{Adopting the perspective of viewing optimization algorithms as dynamical systems, we present} Curvature-Regularized Gradient Dynamics (CRGD), which augments the objective with a smooth penalty on the \rev{negative Hessian eigenvalues}, yielding an augmented cost that serves as an optimization Lyapunov function with \rev{user-selectable convergence rates} to second-order stationary points. \rev{Numerical experiments confirm that CRGD converges to second-order stationary points, even in regimes where gradient descent fails.}
\end{abstract}


\section{Introduction}
\label{sec:introduction}

\IEEEPARstart{N}{onconvex} optimization plays a central role in modern machine learning and control applications, where the speed and reliability of convergence directly impact practical performance. A fundamental difficulty is the accumulation of critical points in high-dimensional objective landscapes. Dauphin et al.~\cite{dauphin_identifying_2014} provided theoretical and empirical evidence that, among these, saddle points, rather than poor local minima, are the dominant obstacle to convergence.

Classical gradient flow $\dot{\bx} = -\nabla J(\bx)$ cannot distinguish saddle points from local minima. While~\cite{lee_gradient_2016} established that gradient descent generically converges to local minimizers, \cite{du_gradient_2017} demonstrated that this convergence can take exponential time from worst-case initializations: starting at distance $\varepsilon_0$ from a saddle with eigenvalue gap $\delta$, the escape time scales as $O(\frac{1}{\delta} \ln \frac{1}{\varepsilon_0})$, rendering gradient descent unreliable when the gap is small. \rev{The idea of exploiting stochastic perturbations to escape saddle points was pioneered in~\cite{ge_escaping_2015} in the context of tensor decomposition, and subsequently refined into perturbed gradient descent (PGD)~\cite{jin_nonconvex_2021}, which provides finite-time escape guarantees under noise injection. Later,~\cite{lee_first-order_2019} generalized the generic convergence result of~\cite{lee_gradient_2016} to a broad class of first-order methods, establishing saddle avoidance as a robust phenomenon, albeit without explicit rate bounds. These stochastic approaches offer only probabilistic guarantees and no user-selectable convergence rate.} In the deterministic domain, \rev{\cite{anandkumar_efficient_2016} developed efficient methods for escaping higher-order saddle points where the Hessian is degenerate, while} Nesterov and Polyak~\cite{nesterov_cubic_2006} introduced cubic regularization of Newton's method, which was later refined in~\cite{carmon_accelerated_2018}, which developed accelerated methods achieving the optimal $O(\bar{\varepsilon}^{-7/4})$ gradient-query complexity for finding $\bar{\varepsilon}$-second-order stationary points (SOSP). Later,~\cite{paternain_newton_2019} proposed a Newton-based method aimed at fast saddle evasion, though requiring noise injection for the full guarantee. 
These methods have considerably advanced the state of the art; however, they share the common limitation that none provide selectable convergence rates, and Newton-type methods suffer from Hessian singularity: the step $\bH^{-1}\bg$ is undefined when $\bH$ is rank-deficient.

The optimization Lyapunov function (OLF) framework~\cite{mudrik_optimization_2025-1} addresses the rate selectability gap by recasting optimization as a feedback design problem: the optimizer is modeled as a single integrator whose control input is synthesized from a Lyapunov certificate, enforcing a user-selected convergence profile exactly. However, the natural quadratic norm of the gradient as a Lyapunov candidate is blind to curvature, and its gradient becomes singular when the Hessian is rank-deficient. Both limitations stem from the choice of Lyapunov function, not the framework itself, motivating the search for a curvature-aware alternative. \rev{This continuous-time, Lyapunov-based perspective on optimization belongs to the broader program of viewing algorithms as dynamical systems~\cite{dorfler_toward_2024}, which has yielded fundamental insights for first-order stationarity but has not yet addressed SOSPs.}

Following the OLF paradigm, we augment the objective with a smooth penalty on the \rev{negative Hessian eigenvalues}. The resulting augmented cost serves as a Lyapunov function whose gradient avoids the Hessian singularity: at a saddle point where the gradient vanishes, the curvature penalty generates a nonzero descent direction, and the dynamics flow away from the saddle. By scheduling the gain through a convergence law, the method enforces exact Lyapunov decay with user-selectable convergence profiles.

The contributions of this letter are as follows. First, we construct an augmented cost $\Phi$ whose critical points coincide with the SOSP of the cost function, ensuring that the dynamics go through saddle points rather than getting trapped. Next, we show that $\Phi$ serves as a Lyapunov function admitting user-selected convergence laws, thereby yielding convergence to SOSP. We further prove that any spurious critical points introduced by the curvature penalty are themselves saddle points of $\Phi$ for penalty weights below a landscape-dependent threshold, ensuring they cannot be limit points of the dynamics. Finally, we validate the method on numerical examples, \rev{demonstrating 100\% convergence to SOSPs under all four convergence laws, even when gradient descent fails}. \rev{To the best of the authors' knowledge, CRGD is the first continuous-time deterministic method that provides convergence to second-order stationary points on general nonconvex landscapes with user-selectable convergence laws.}

Table~\ref{tab:comparison} summarizes the positioning of the proposed method relative to existing approaches; notably, CRGD is the only method that combines deterministic guarantees with selectable convergence rates.
The remainder of this letter is organized as follows. Section~\ref{sec:preliminaries} formulates the problem, presents the OLF framework, and states the standing assumptions. Section~\ref{sec:contributions} develops the CRGD method and establishes the convergence guarantees. Numerical validation is presented in Sec.~\ref{sec:results}, and concluding remarks are offered in Sec.~\ref{sec:conclusion}.

\begin{table}[htbp]
    \centering
    \caption{Comparison of saddle evasion methods.}
    \label{tab:comparison}
    \renewcommand{\arraystretch}{1.15}
    \begin{tabular}{@{}lcccc@{}}
        \toprule
        Method & Fully Det. & \rev{Selectable} & Cost & \rev{Info} \\
        \midrule
        Gradient Descent~\cite{lee_gradient_2016} & \checkmark & --- & $O(n)$ & \rev{$\nabla J$} \\
        PGD~\cite{jin_nonconvex_2021} & --- & --- & $O(n)$ & \rev{$\nabla J$} \\
        Cubic Reg.~\cite{nesterov_cubic_2006} & \checkmark & --- & $O(n^3)$ & \rev{$\nabla J$, $\nabla^2 J$} \\
        Newton-based~\cite{paternain_newton_2019} & --- & --- & $O(n^3)$ & \rev{$\nabla J$, $\nabla^2 J$} \\
        CRGD (Proposed) & \checkmark & \checkmark & $O(n^3)$ & \rev{$\nabla J$, $\lambda_i(\nabla^2 J)$} \\
        \bottomrule
    \end{tabular}
    \\[4pt]
    \raggedright\footnotesize Fully Det.: deterministic guarantee, no stochastic perturbation. \rev{Selectable: user-selectable convergence rate}. Cost: per-step complexity\rev{. Info: oracle information per step}.
\end{table}

\section{Problem Statement and Preliminaries}
\label{sec:preliminaries}

We consider the unconstrained optimization problem
\begin{equation} \label{eq:problem}
    \min_{\bx \in \Rn} J(\bx),
\end{equation}
where $J : \Rn \to \mathbb{R}$ is the objective function with gradient $\bg(\bx) = \nabla J(\bx)$ and Hessian $\bH(\bx) = \nabla^2 J(\bx)$.

\begin{definition}[Strict Saddle Point] \label{def:saddle}
A point $\bx^* \in \Rn$ is a strict saddle point if $\bg(\bx^*) = \mathbf{0}$ and $\bH(\bx^*)$ has at least one strictly negative eigenvalue.
\end{definition}

The complementary notion identifies critical points satisfying second-order necessary conditions. The set of SOSP is
\begin{equation} \label{eq:sosp}
    \cX^* := \{ \bx \in \Rn \mid \bg(\bx) = \mathbf{0},\ \bH(\bx) \succeq 0 \}.
\end{equation}
This set includes all local minima but strictly excludes strict saddle points. Our goal is to design a feedback law that drives the state to $\cX^*$ at a user-selectable rate.

\subsection{Optimization Lyapunov Functions}
\label{sec:olf}

The OLF framework \textit{synthesizes} the optimizer from a Lyapunov certificate rather than proposing an algorithm and analyzing it post hoc. The state is governed by a single integrator
\begin{equation} \label{eq:single_integrator}
    \dot{\bx} = \bu,
\end{equation}
where $\bu \in \Rn$ is the input to be designed. 
\rev{In fact, an integrator structure is a necessary condition for convergence~\cite{scherer_convex_2021}.}
The key idea is to choose a candidate Lyapunov function $V(\bx) \geq 0$ whose zero set $\{V = 0\}$ encodes the desired optimality conditions, and then design $\bu$ so that $V$ decreases at a prescribed rate. Specifically, one selects a convergence law $\sigma(V, t) \geq 0$ and requires the Lyapunov function to satisfy
\begin{equation} \label{eq:lyap_decay}
    \dot{V} = -\sigma(V, t).
\end{equation}
The framework is agnostic to the specific form of $\sigma$; four well-studied instances from the finite-time stability literature~\cite{bhat_finite-time_2000, polyakov_nonlinear_2012, song_time-varying_2017} are:
\begin{description}
    \item[Exponential~\cite{khalil_nonlinear_2002}:] 
    \begin{equation}
        \sigma = cV,\quad c > 0.
    \end{equation}
    Yields asymptotic decay $V(t) = V(0)\,e^{-ct}$.
    \item[Finite-time~\cite{bhat_finite-time_2000}:]
    \begin{equation}
        \sigma = c V^\alpha, \quad\alpha \in (0,1).
    \end{equation}
    Drives $V$ to zero in finite time $T_s$ that depends on $V(0)$.
    \item[Fixed-time~\cite{polyakov_nonlinear_2012}:] 
    \begin{equation}
        \sigma = c_1 V^\alpha + c_2 V^p, \quad0 < \alpha < 1 < p.
    \end{equation}
    Settling time is bounded independently of $V(0)$.
    \item[Prescribed-time~\cite{song_time-varying_2017}:] 
    \begin{equation}
        \sigma = \mu_T\frac{V} {T - t}, \quad\mu_T > 1.
    \end{equation}
    Enforces $V(T) = 0$ exactly at the user-specified instant~$T$.
\end{description}
All four laws satisfy $\sigma(0, t) = 0$, ensuring that $V = 0$ is an equilibrium.

Having specified the convergence law, the feedback law follows by differentiating $V$ along~\eqref{eq:single_integrator} and equating with~\eqref{eq:lyap_decay}:
\begin{equation} \label{eq:olf_feedback}
    \bu = -\frac{\sigma(V(\bx), t)}{\| \nabla_{\bx} V(\bx) \|^2} \nabla_{\bx} V(\bx).
\end{equation}
Direct substitution confirms $\dot{V} = -\sigma(V, t)$ exactly, which is the defining feature of the OLF framework: the Lyapunov decay is enforced as an identity, not merely as an upper bound, and the convergence profile is selected by the user through the choice of~$\sigma$.

The feedback law~\eqref{eq:olf_feedback} is well-defined provided $\nabla_{\bx} V \neq \mathbf{0}$ whenever $V > 0$; if $\nabla_{\bx} V = \mathbf{0}$ at a point where $V > 0$, the gain becomes singular, and the feedback law breaks down. The design challenge is to construct a Lyapunov function $V$ whose zero set coincides with $\cX^*$ and whose gradient does not vanish prematurely, that is, a function that distinguishes saddle points from local minima at the level of both its zero set and its gradient.

A natural first-order candidate is $V_1 = \frac{1}{2}\|\bg\|^2$, which vanishes at all critical points. However, $V_1$ is blind to curvature: it treats saddle points and local minima identically. Moreover, its gradient $\nabla_{\bx} V_1 = \bH \bg$ passes through the Hessian, so whenever $\bg \in \ker(\bH)$, a condition that holds along certain trajectories approaching rank-deficient critical points, the gradient vanishes while $V_1 > 0$, violating the regularity condition and rendering~\eqref{eq:olf_feedback} ill-defined.

\subsection{Standing Assumptions}
\label{sec:assumptions}

The following conditions on the objective $J$ are assumed throughout this letter.
\begin{assumption}[Smoothness] \label{ass:smooth}
$J \in C^4(\Rn)$.
\end{assumption}
This ensures that $\bH(\bx)$ is twice continuously differentiable, 
as required by the chain rule in Prop.~\ref{prop:regularity}.

\begin{assumption}[Eigenvalue Simplicity] \label{ass:landscape}
At every strict saddle point $\bx^*$ of $J$, the minimum eigenvalue $\lambda_{\min}(\bH(\bx^*))$ is a simple eigenvalue and satisfies $\nabla_{\bx} \lambda_{\min}(\bx^*) \neq \mathbf{0}$.
\end{assumption}
Simplicity of $\lambda_{\min}$ holds generically in the space of real symmetric matrices~\cite{magnus_differentiating_1985}, and when it holds, $\lambda_{\min}(\bx)$ is differentiable by first-order eigenvalue perturbation theory. The condition $\nabla_{\bx} \lambda_{\min} \neq \mathbf{0}$ is itself generic: $\bg = \mathbf{0}$ and $\nabla_{\bx} \lambda_{\min} = \mathbf{0}$ simultaneously impose $2n$ equations on $n$ unknowns. \rev{Simplicity is required only at strict saddles of $J$ (used in Prop.~\ref{prop:saddle_destab}); Prop.~\ref{prop:regularity} handles multiplicity elsewhere.}

\begin{assumption}[Compact Sublevel Sets] \label{ass:compact}
The sublevel sets $\{ \bx \in \Rn : J(\bx) \leq c \}$ are compact for every $c \in \mathbb{R}$.
\end{assumption}
This standard regularity condition ensures trajectories remain bounded.

\section{Curvature-Regularized Gradient Dynamics}
\label{sec:contributions}

\subsection{The Augmented Cost}

The analysis in Sec.~\ref{sec:olf} identified curvature blindness as the root limitation of the squared norm of the gradient as the OLF. To encode second-order information, we augment the objective with a smooth penalty on the \rev{negative Hessian spectrum}. \rev{Let $\lambda_1(\bx), \ldots, \lambda_n(\bx)$ denote the eigenvalues of $\bH(\bx)$ and define the smooth negative-part function 
\begin{equation}
    \psi_\varepsilon(\lambda) = \frac{1}{2}(\sqrt{\lambda^2 + \varepsilon^2} - \lambda)
\end{equation} 
for a fixed smoothing parameter $\varepsilon > 0$, so that $\psi_\varepsilon(\lambda) \to \max(0, -\lambda)$ as $\varepsilon \to 0^+$.} We define the augmented cost
\begin{equation} \label{eq:Phi}
    \Phi(\bx) = J(\bx) + P(\bx) = J(\bx) + 
    \rev{\tfrac{\beta^2}{2} \sum_{i=1}^{n} \psi_\varepsilon\bigl(\lambda_i(\bx)\bigr)^2},
\end{equation}
where $P(\bx)$ is the curvature penalty and $\beta > 0$ is a curvature penalty weight. \rev{The curvature penalty is nonnegative and, for small~$\varepsilon$, is negligible when $\bH(\bx) \succ 0$ while $P(\bx) \approx \frac{\beta^2}{2}\lambda_{\min}^2$ at strict saddle points where $\lambda_{\min}$ is the unique negative eigenvalue, thus} the penalty lifts the cost landscape precisely where second-order optimality is violated.

\subsection{Regularity and Gradient}

\rev{The curvature penalty, $P(\bx)$, is a spectral function, which is a symmetric function of the eigenvalues composed with the Hessian map. Since $\psi_\varepsilon$ is analytic for $\varepsilon > 0$, the spectral mapping theorem for analytic symmetric functions~\cite{lewis_derivatives_1996} yields the following global regularity.}

\begin{proposition}[Regularity of $\Phi$] \label{prop:regularity}
\rev{Under Assumption~\ref{ass:smooth}, $\Phi \in C^2(\Rn)$ for any $\varepsilon > 0$.}
\end{proposition}

\begin{proof}
\rev{The function $g(\lambda) = \frac{1}{2}\psi_\varepsilon(\lambda)^2$ is analytic for $\varepsilon > 0$, so $f(\lambda_1, \ldots, \lambda_n) = \sum_i g(\lambda_i)$ is an analytic symmetric function. By the analyticity theorem of~\cite{tsing_analyticity_1994}, any analytic symmetric function of eigenvalues induces an analytic (hence $C^\infty$) map on the space of real symmetric matrices; Lewis~\cite{lewis_derivatives_1996} provides explicit derivative formulae for the resulting spectral function. Since $\bH = \nabla^2 J \in C^2(\Rn; \mathbb{S}^n)$ under Assumption~\ref{ass:smooth}, the chain rule gives $P \in C^2(\Rn)$ and hence $\Phi = J + P \in C^2(\Rn)$.}
\end{proof}

\rev{For the non-smoothed penalty, i.e., $\varepsilon = 0$ thus rendering $\psi_0 = \max(0,-\lambda)$, the same spectral mapping argument yields $\Phi \in C^{1,1}(\Rn)$ globally; the smoothing $\varepsilon > 0$ upgrades this to $C^2$, as required for the Hessian analysis in Prop.~\ref{prop:spurious}.}

\rev{Let $\bu_i(\bx)$ denote a unit eigenvector of $\bH(\bx)$ associated with $\lambda_i$ and define the curvature sensitivity vectors $\bw_i \in \Rn$ with components}
\begin{equation} \label{eq:w_min}
    \rev{(w_i)_j = \bu_i^\top \frac{\partial \bH}{\partial x_j} \bu_i, \quad j = 1, \ldots, n;}
\end{equation}
\rev{by eigenvalue perturbation theory~\cite{magnus_differentiating_1985}, $\bw_i = \nabla_{\bx} \lambda_i$ when $\lambda_i$ is simple; at eigenvalue crossings the individual $\bw_i$ are basis-dependent, but $\psi_\varepsilon(\lambda)\psi'_\varepsilon(\lambda)$ is constant within each eigenspace, so the basis-dependent contributions combine into a basis-independent quantity and~\eqref{eq:grad_Phi} is well-defined regardless of multiplicity. We write $\bw_{\min}$ for the vector associated with $\lambda_{\min}$.}

\rev{The gradient of $\Phi$ is}
\begin{equation} \label{eq:grad_Phi}
    \nabla \Phi = \bg + \rev{\beta^2 \sum_{i=1}^{n} \psi_\varepsilon(\lambda_i)\,\psi_\varepsilon'(\lambda_i) \, \bw_i}.
\end{equation}
\rev{The curvature correction is driven by the negative part of the spectrum, since $\psi_\varepsilon(\lambda)\,\psi_\varepsilon'(\lambda)$ is negligible for $\lambda \gg \varepsilon$. At points where $\lambda_{\min}$ is the sole significantly negative eigenvalue, the sum reduces to the single-term approximation $\nabla \Phi \approx \bg - \beta^2 |\lambda_{\min}| \, \bw_{\min}$.} At a strict saddle point where $\bg = \mathbf{0}$ and $\lambda_{\min} < 0$:
\begin{equation} \label{eq:grad_Phi_saddle}
    \nabla \Phi \big|_{\bg = \mathbf{0}} = \rev{\beta^2 \sum_{i} \psi_\varepsilon(\lambda_i)\,\psi_\varepsilon'(\lambda_i) \, \bw_i} \neq \mathbf{0},
\end{equation}
provided $\bw_{\min} \neq \mathbf{0}$ (a generic condition; see Assumption~\ref{ass:landscape})\rev{, since the $\lambda_{\min}$ term dominates the sum}. Thus, saddle points of $J$ are not critical points of $\Phi$: the CRGD dynamics go through them rather than getting trapped. \rev{Importantly, the curvature information enters through the $\bw_i$ rather than through a Hessian-gradient product, avoiding the singularity that plagues the $V_1 = \frac{1}{2}\|\bg\|^2$ candidate.}

\subsection{Dynamics}
\label{sec:dynamics}

Applying the OLF feedback law~\eqref{eq:olf_feedback} with the Lyapunov function $V = \Phi - \Phi^*$, where $\Phi^* := \inf_{\bx \in \cX^*} \Phi(\bx)$, yields the CRGD control law
\begin{equation} \label{eq:crgd}
    \bu = -\frac{\sigma(\Phi(\bx) - \Phi^*, t)}{\| \nabla \Phi(\bx) \|^2} \nabla \Phi(\bx).
\end{equation}
When $\Phi^*$ is not known a priori, \rev{any computable lower bound $\tilde{\Phi}^* \leq \Phi^*$ may be used; $\tilde{\Phi}^* = 0$ is valid whenever $J \geq 0$ (e.g., matrix factorization, in Sec.~\ref{sec:results} below). For general bounded-below objectives, a tighter bound must be supplied; choosing $\tilde{\Phi}^* > \Phi^*$ causes premature termination at non-stationary points.} 
The well-posedness of~\eqref{eq:crgd} requires the following regularity condition on the augmented cost.

\begin{assumption}[Well-Posedness] \label{ass:nondegen}
$\nabla \Phi(\bx) \neq \mathbf{0}$ whenever $\Phi(\bx) > \Phi^*$ and $\bx$ is not a local minimizer of $\Phi$.
\end{assumption}

This condition ensures the feedback law~\eqref{eq:crgd} is well-defined along the trajectory. Assumption~\ref{ass:landscape} provides an essential ingredient: at strict saddle points of $J$, \rev{Prop.~\ref{prop:saddle_destab} ensures $\nabla \Phi \neq \mathbf{0}$}. Assumption~\ref{ass:nondegen} extends this to the full trajectory, requiring that $\bg$ and the curvature correction in~\eqref{eq:grad_Phi} do not cancel along the path to convergence. \rev{Such cancellation is a non-generic condition of codimension~$n$ that can be verified on any specific landscape.} More broadly, Assumption~\ref{ass:nondegen} is a non-singularity condition on the critical points of~$\Phi$, analogous to the hyperbolicity requirement for equilibria of linear systems~\cite[Ch.~2]{khalil_nonlinear_2002}. Competing saddle-escape methods impose analogous structural conditions:~\cite{paternain_newton_2019} requires all saddle eigenvalues to be bounded away from zero, while~\cite{jin_nonconvex_2021} provides only probabilistic guarantees.

\begin{remark}[Role of $\beta$ and Computational Cost] \label{rem:beta_cost}
The curvature weight $\beta$ must be positive to activate the curvature correction (Prop.~\ref{prop:saddle_destab}) and satisfy $\beta < \beta^*$ to ensure spurious critical points are saddles of $\Phi$ (Prop.~\ref{prop:spurious}). \rev{Conversely, $\beta$ must be sufficiently large for the penalty gradient to dominate the objective gradient near saddle points; otherwise $\Phi$ admits slow progress with residual negative curvature.} 
The dominant per-step cost is the $O(n^3)$ eigendecomposition of $\bH$\rev{; since the dominant contributions to $\nabla P$ come from the negative eigenvalues, Lanczos iteration can
reduce this to $O(n^2)$~\cite{carmon_accelerated_2018}.}
\end{remark}

\subsection{Convergence Analysis}

We now establish the main theoretical results. The first proposition shows that saddle points of $J$ are not equilibria of the CRGD dynamics.

\begin{proposition}[Saddle Points are Non-Critical for $\Phi$] \label{prop:saddle_destab}
Let $\bx^*$ be a strict saddle point of $J$ with $\lambda_{\min}(\bx^*) < 0$, and suppose $\bw_{\min}(\bx^*) \neq \mathbf{0}$. Then $\nabla \Phi(\bx^*) \neq \mathbf{0}$, and $\bx^*$ is not an equilibrium of~\eqref{eq:crgd}.
\end{proposition}

\begin{proof}
At $\bx^*$, $\bg(\bx^*) = \mathbf{0}$ and $\lambda_{\min}(\bx^*) < 0$, so~\eqref{eq:grad_Phi_saddle} gives $\nabla \Phi(\bx^*) = \rev{\beta^2 \sum_i \psi_\varepsilon(\lambda_i)\,\psi_\varepsilon'(\lambda_i)\,\bw_i}$\rev{. Since every coefficient $\psi_\varepsilon(\lambda_i)\psi_\varepsilon'(\lambda_i)$ is strictly negative, the $\lambda_{\min}$ term contributes a nonzero vector proportional to~$\bw_{\min}(\bx^*) \neq \mathbf{0}$; exact cancellation by the remaining terms is a non-generic condition (codimension~$n$), and Assumption~\ref{ass:nondegen} provides the formal guarantee along trajectories}.
\end{proof}

\rev{The dynamics~\eqref{eq:crgd} may, however, introduce spurious critical points of $\Phi$: points where the gradient and curvature correction in~\eqref{eq:grad_Phi} balance to give $\nabla \Phi = \mathbf{0}$, with $\lambda_{\min} < 0$.} The following proposition establishes that these are not attractors.

\begin{proposition}[Spurious Points are Saddles of $\Phi$] \label{prop:spurious}
Let $\bar{\bx}$ be a critical point of $\Phi$ with $\lambda_{\min}(\bar{\bx}) < 0$. 
Then there exists $\beta^*$, such that for any $0 < \beta < \beta^*$, $\nabla^2 \Phi(\bar{\bx})$ has at least one strictly negative eigenvalue.
\end{proposition}

\begin{proof}
At a spurious critical point $\bar{\bx}$\rev{,} $\nabla \Phi = \mathbf{0}$ with $\lambda_{\min} < 0$\rev{; let $\bu_{\min}$ be any unit eigenvector of $\bH(\bar{\bx})$ for $\lambda_{\min}$. By Prop.~\ref{prop:regularity}, $\nabla^2 \Phi(\bar{\bx})$ is well-defined regardless of multiplicity}. \rev{Evaluating the quadratic form of $\nabla^2 \Phi$ along $\bu_{\min}$:}
\begin{align}
    \bu_{\min}^\top &\nabla^2 \Phi \, \bu_{\min} = \underbrace{\lambda_{\min}}_{<\,0} + \beta^2 \underbrace{\bu_{\min}^\top \nabla^2_{\bx} P \, \bu_{\min} / \beta^2}_{=:\,C(\bar{\bx})}\rev{,} \notag
\end{align}
\rev{where $C(\bar{\bx})$ collects the curvature penalty's second-order contribution along~$\bu_{\min}$, including the dominant $\lambda_{\min}$ term and corrections from the remaining eigenvalues.}
At $\beta = 0$ the quadratic form equals $\lambda_{\min} < 0$. By continuity, for all $\beta^2 < |\lambda_{\min}| / |C(\bar{\bx})|$ the quadratic form remains strictly negative (if $C(\bar{\bx}) = 0$, \rev{it is negative} for all $\beta > 0$). Taking the infimum over all spurious critical points yields the threshold
\begin{equation} \label{eq:beta_star}
    \beta^* := \sqrt{ \frac{\inf_{\bar{\bx} \in \cS} |\lambda_{\min}(\bar{\bx})|}{\sup_{\bar{\bx} \in \cS} \rev{|C(\bar{\bx})|}} }\rev{,}
\end{equation}
with $\cS$ denoting the set of spurious critical points of $\Phi$ (those with $\lambda_{\min} < 0$)\rev{.}
The set $\cS$ depends implicitly on $\beta$; however, for $\beta$ sufficiently small, $\cS$ is empty \rev{ (via Prop.~\ref{prop:saddle_destab} and Assumption~\ref{ass:compact})} and $\beta^*$ is vacuously infinite, so the threshold becomes relevant only when spurious critical points exist.
\end{proof}

\begin{remark} \label{rem:lyap_spurious}
Since the dynamics~\eqref{eq:crgd} enforce $\dot{\Phi} = -\sigma < 0$, the trajectory cannot converge to a strict saddle point of $\Phi$. To see this, note that the stable manifold theorem~\cite{khalil_nonlinear_2002} implies that the set of initial conditions whose trajectories converge to a saddle $\bar{\bx}$ of $\Phi$ is a smooth manifold of dimension at most $n - 1$, and hence has Lebesgue measure zero in $\Rn$. Consequently, for generic initial conditions all limit points are local minima of $\Phi$, which satisfy $\lambda_{\min} \geq 0$ and hence are SOSPs of $J$.
\end{remark}

The preceding results combine into the main convergence guarantee.

\begin{theorem}[\rev{Convergence with Selectable Rate}] \label{thm:convergence}
Let Assumptions~\ref{ass:smooth}--\ref{ass:compact} and~\ref{ass:nondegen} hold and let $0 < \beta < \beta^*$. \rev{The CRGD dynamics~\eqref{eq:crgd} satisfy $\dot{\Phi} = -\sigma(\Phi - \Phi^*, t)$ on $\{\Phi > \Phi^*\}$, and $\bx(t)$ converges at the prescribed rate to a limit point $\bx^*$ with $\bg(\bx^*) = O(\beta^2 n \varepsilon)$ and $\bH(\bx^*) \succeq 0$, recovering $\cX^*$ as $\varepsilon \to 0^+$}.
\end{theorem}

\begin{proof}
\rev{By Prop.~\ref{prop:regularity}, $\Phi \in C^2(\Rn)$, so $\nabla\Phi$ is $C^1$ and the right-hand side of~\eqref{eq:crgd} is locally Lipschitz wherever $\nabla \Phi \neq \mathbf{0}$.} Assumption~\ref{ass:nondegen} ensures $\| \nabla \Phi \| > 0$ on $\{ \Phi > \Phi^* \}$, so the dynamics are well-defined. Direct computation gives $\dot{\Phi} = \nabla \Phi^\top \bu = -\sigma(\Phi - \Phi^*, t)$. Assumption~\ref{ass:compact} guarantees the trajectory exists for all $t \geq 0$ (or $t \in [0, T)$ for prescribed-time). The convergence rates follow from the explicit solutions of $\dot{V} = -\sigma(V, t)$ with $V = \Phi - \Phi^*$; the reader is referred to~\cite{bhat_finite-time_2000, polyakov_nonlinear_2012, song_time-varying_2017} for details.
It remains to show that limit points are SOSPs of $J$. Since $\dot{\Phi} \leq 0$, the augmented cost is nonincreasing along the trajectory, so any limit point $\bx^*$ is a critical point of $\Phi$. By Prop.~\ref{prop:saddle_destab}, strict saddle points of $J$ are not critical points of $\Phi$. By Prop.~\ref{prop:spurious}, spurious critical points with $\lambda_{\min} < 0$ are saddle points of $\Phi$ for $0 < \beta < \beta^*$. Since $\dot{\Phi} = -\sigma < 0$ wherever $\Phi > \Phi^*$, the trajectory cannot converge to a saddle point of $\Phi$ (Remark~\ref{rem:lyap_spurious}; Assumption~\ref{ass:nondegen} excludes the measure-zero set of pathological initializations). Therefore, $\bx^*$ is a local minimizer of $\Phi$ with $\lambda_{\min}(\bx^*) \geq 0$\rev{; since $\psi_\varepsilon(\lambda) \leq \varepsilon/2$ and $|\psi_\varepsilon'(\lambda)| \leq 1/2$ for $\lambda \geq 0$, $\nabla\Phi(\bx^*) = \mathbf{0}$ gives $\bg(\bx^*) = -\nabla P(\bx^*)$ with $\|\bg(\bx^*)\| = O(\beta^2 n \varepsilon)$, vanishing as $\varepsilon \to 0^+$}.
\end{proof}

Theorem~\ref{thm:convergence} guarantees convergence \rev{to the limit point $\bx^*$}, but leaves open the behavior at the instant of arrival, where the feedback law~\eqref{eq:crgd} has a removable singularity. The following proposition shows that setting $\bu = \mathbf{0}$ at the equilibrium yields a continuous trajectory.
\begin{proposition}[Equilibrium Extension] \label{prop:wellposed}
Define the extended control $\bu = \mathbf{0}$ on $\{ \Phi = \Phi^* \}$. Then the resulting trajectory is continuous at any non-degenerate equilibrium (where $\nabla^2 \Phi(\bx^*) \succ 0$) under any convergence law $\sigma$.
\end{proposition}
\begin{proof}
The arc length $\int_0^{T_s} \| \bu \| \, \mathrm{d}t = \int_{\Phi^*}^{\Phi(0)} \mathrm{d}\Phi / \| \nabla \Phi \|$ (obtained via $\sigma \, \mathrm{d}t = -\mathrm{d}\Phi$) is finite near a non-degenerate minimum, since $\| \nabla \Phi \| \geq c_\ell (\Phi - \Phi^*)^{1/2}$ gives a bound of $\frac{2}{c_\ell}(\Phi(0) - \Phi^*)^{1/2}$. Finite arc length ensures $\bx^* := \lim_{t \to T_s} \bx(t)$ exists, and $\| \bu \| = \sigma / \| \nabla \Phi \| \to 0$ ensures continuity.
\end{proof}

\section{Numerical Validation}
\label{sec:results}

We illustrate the framework on two problem classes using MATLAB's \texttt{ode15s} solver with relative tolerance $10^{-10}$. In all experiments, the Lyapunov function is $V = \Phi - \tilde{\Phi}^*$ with $\tilde{\Phi}^* = 0$ as described in Sec.~\ref{sec:dynamics}\rev{, the smoothing parameter is $\varepsilon = 10^{-6}$,} and the numerical implementation uses the regularized form $\bu = -\sigma \nabla \Phi / (\| \nabla \Phi \|^2 + \epsilon_r)$ with $\epsilon_r = 10^{-12}$\rev{, introducing a negligible $O(\epsilon_r / \| \nabla \Phi \|^2)$ deviation from the exact decay}. All simulations compare CRGD against standard gradient descent $\dot{\bx} = -\nabla J$. Unless otherwise noted, the convergence law parameters are: $c = 2$ (exponential); $c = 2$, $\alpha = 0.5$ (finite-time); $c_1 = c_2 = 1$, $\alpha = 0.5$, $p = 1.5$ (fixed-time); and $T = 0.1$\,s, $\mu_T = 2$ (prescribed-time).

\subsection{2D Example: Three-Fold Potential}

Consider the three-fold symmetric landscape
\begin{equation} \label{eq:threefold}
    J(\bx) = \tfrac{1}{2}\|\bx\|^2 + \tfrac{1}{4}\|\bx\|^4 - \eta(x_1^3 - 3 x_1 x_2^2),
\end{equation}
with $\eta = 0.7$. This potential has seven critical points: the global minimum at the origin, three saddle points at radius $r_s \approx 0.73$ ($\lambda_{\min} \approx -0.47$), and three outer local minima at radius $r_m \approx 1.37$ ($\lambda_{\min} \approx 0.88$).

We initialize at $\bx_0 = [r_s, \, 10^{-4}]^\top$, on the stable manifold of a saddle, and set $\beta = 1$. Figure~\ref{fig:trajectories} shows that gradient descent converges to the saddle ($J = 0.065$), trapped along the stable manifold, while CRGD escapes to the outer local minimum ($J = 0.019$). A grid evaluation of $\|\nabla \Phi\|$ on a $500 \times 500$ mesh confirms no spurious critical points of $\Phi$ at $\beta = 1$. 

\begin{figure}[t]
    \centering
    \includegraphics[width=0.9\linewidth]{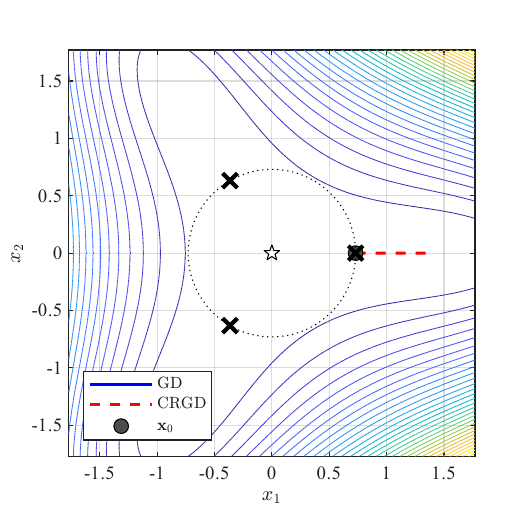}
    \caption{Three-fold potential: GD converges to the saddle (\textsf{x} marker), while CRGD escapes to the outer local minimum.}
    \label{fig:trajectories}
\end{figure}

The GD curve in Fig.~\ref{fig:rates} stalls near $\Phi \approx 0.065$. The four CRGD curves each produce a distinct decay profile closely tracking the theoretical ODE solution (dashed) until the plateau at $\Phi_{\mathrm{eq}} \approx 0.019$, the cost at the outer local minimum where the curvature penalty vanishes.

\begin{figure}[t]
    \centering
    \includegraphics[width=0.9\linewidth]{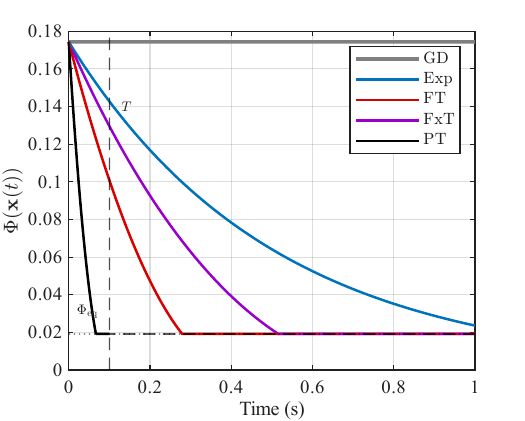}
    \caption{Augmented cost $\Phi(\bx(t))$ on the three-fold potential. GD (gray) stalls near the saddle; CRGD curves track theoretical solutions (dashed) until $\Phi_{\mathrm{eq}} \approx 0.019$, with the prescribed-time law converging at $t = T = 0.1$\,s.}
    \label{fig:rates}
\end{figure}

\subsection{High-Dimensional Example: Matrix Factorization}

We next consider the rank-one symmetric matrix factorization problem
\begin{equation} \label{eq:smf}
    J(\bx) = \tfrac{1}{4}\|\bx\bx^\top - \mathbf{M}^*\|_F^2,
\end{equation}
where $\mathbf{M}^* \in \mathbb{S}^n$ has eigenvalues $\lambda_1 > \lambda_2 > \cdots > \lambda_n \geq 0$. This nonconvex objective, whose strict saddle landscape was characterized by Li et al.~\cite{li_symmetry_2019}, arises in low-rank matrix recovery and related applications. The gradient is $\nabla J = (\|\bx\|^2 \mathbf{I} - \mathbf{M}^*)\bx$ and the Hessian is $\nabla^2 J = \|\bx\|^2 \mathbf{I} + 2\bx\bx^\top - \mathbf{M}^*$. The global minimizers are $\bx^* = \pm\sqrt{\lambda_1}\,\mathbf{v}_1$, and each $\bx_s = \pm\sqrt{\lambda_i}\,\mathbf{v}_i$ for $i \geq 2$ is a saddle point with $\lambda_{\min}(\nabla^2 J(\bx_s)) = \lambda_i - \lambda_1$. Setting $\lambda_1 = 1$ and $\lambda_2 = 1 - \delta$ produces a tunable eigenvalue gap~$\delta$ at the dominant saddle.

The curvature sensitivity vector admits the closed-form expression $\bw_{\min} = 2\bx + 4(\bx^\top\bu_{\min})\bu_{\min}$, which is $O(n)$ to evaluate and independent of $\mathbf{M}^*$. All experiments use $n = 50$, $\beta = 1$ with the exponential law $\sigma = cV$, $c = 2$. Since $J \geq 0$, the lower bound $\tilde{\Phi}^* = 0$ is valid for this example.

\subsubsection{\rev{Convergence Profiles Under Four Laws}}
\rev{
Figure~\ref{fig:profiles} displays the Lyapunov function $V(t) = \Phi(\bx(t)) - \Phi_{\mathrm{eq}}$ on a semilogarithmic scale for all four convergence laws, initialized from a random unit-sphere direction with $\delta = 0.01$. Each trajectory tracks its theoretical decay envelope (dashed) until reaching the equilibrium value $\Phi_{\mathrm{eq}}$ at the global minimizer, where the curvature penalty vanishes.}

\begin{figure}[t]
    \centering
    \includegraphics[width=0.9\linewidth]{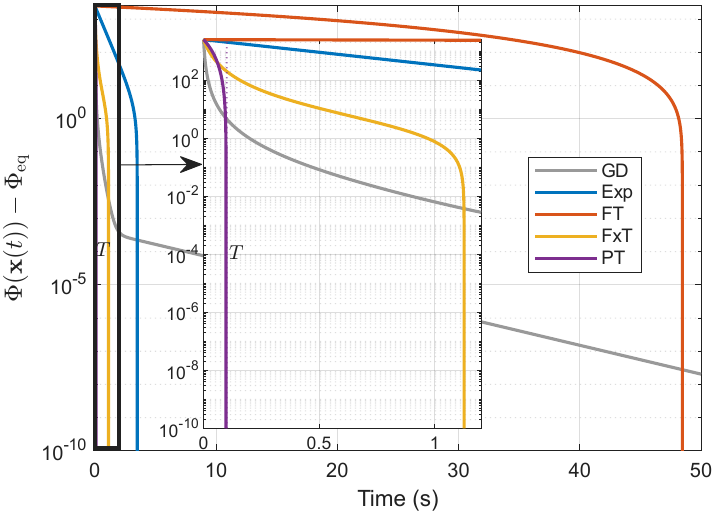}
    \caption{\rev{OLF $V(t) = \Phi(\bx(t)) - \Phi_{\mathrm{eq}}$ on matrix factorization ($n = 50$, $\delta = 0.01$). All four CRGD laws converge to tolerance $10^{-10}$.}}
    \label{fig:profiles}
\end{figure}

\subsubsection{\rev{Monte Carlo Study}}
\rev{
Table~\ref{tab:mc} reports SOSP convergence rates over 2000 random unit-sphere initializations, with horizon $T = 10$\,s for CRGD and $T = 1000$\,s for gradient descent. CRGD achieves 100\% across all configurations; gradient descent degrades to 53\% at small gaps despite the much longer horizon, as trajectories near narrow-gap saddle points remain trapped. A sensitivity study ($n = 50$, $\delta = 0.01$) confirms robustness for all $\beta \leq 125$, bracketing $\beta^* \in (125, 126)$ consistent with Prop.~\ref{prop:spurious}. Numerical inspection of all trajectories confirms $\|\nabla\Phi\| > 0$ throughout, satisfying Assumption~\ref{ass:nondegen} on this benchmark.} 

\begin{table}[t]
    \centering
    \caption{\rev{Monte Carlo SOSP convergence rates (\%) over 2000 trials per gap $\delta$; horizons $T = 1000$\,s (GD), $T = 10$\,s (CRGD); CRGD with exponential law, $\beta = 1$.}}
    \label{tab:mc}
    \renewcommand{\arraystretch}{1.15}
    \begin{tabular}{@{}lcccc@{}}
        \toprule
        $\delta$ & \rev{0.100} & \rev{0.005} & \rev{0.002} & \rev{0.001} \\
        \midrule
        GD   & \rev{100} & \rev{88}  & \rev{53}  & \rev{53}  \\
        CRGD & \rev{100} & \rev{100} & \rev{100} & \rev{100} \\
        \bottomrule
    \end{tabular}
\end{table}

\section{Conclusion}
\label{sec:conclusion}

We have presented Curvature-Regularized Gradient Dynamics, a deterministic method that extends the control-centric OLF framework to enforce second-order optimality via a smooth curvature penalty on the \rev{negative Hessian eigenvalues}. The method eliminates saddle points as equilibria while preserving all local minima, and the augmented cost serves as a Lyapunov function with exact decay $\dot{\Phi} = -\sigma(\Phi, t)$, yielding user-selectable convergence rates to SOSPs. \rev{Numerical experiments confirm 100\% convergence to SOSPs under all four convergence laws, contributing to the systems-theoretic perspective on optimization~\cite{dorfler_toward_2024} by extending the continuous-time Lyapunov framework from first-order to second-order stationarity.} The control-centric perspective developed here, which constructs Lyapunov certificates that encode second-order optimality by design, may inform the development of optimization algorithms for broader classes of nonconvex problems arising in control and learning.

\bibliographystyle{IEEEtran}
\bibliography{references}

\end{document}